\documentclass[12pt]{article}
\usepackage{amssymb,amsmath,amsthm,amsfonts,amscd}
\usepackage{graphicx}
\usepackage{color}
\usepackage[all]{xy}
\usepackage{hyperref}
\hypersetup{
    colorlinks=true,
    linkcolor=blue,
    citecolor=blue,
    filecolor=blue,
    urlcolor=blue
}
\newtheorem{myth}{Theorem}[section]
\newtheorem{mylem}{Lemma}[section]

\newtheorem{mydef}{Definition}[section]

\def\Xint#1{\mathchoice
    {\XXint\displaystyle\textstyle{#1}}%
    {\XXint\textstyle\scriptstyle{#1}}%
    {\XXint\scriptstyle\scriptscriptstyle{#1}}%
    {\XXint\scriptscriptstyle\scriptscriptstyle{#1}}%
      \!\int}
\def\XXint#1#2#3{{\setbox0=\hbox{$#1{#2#3}{\int}$}
    \vcenter{\hbox{$#2#3$}}\kern-.5\wd0}}
\def\dashint{\Xint-}

\def\YYint#1#2#3{{\setbox0=\hbox{$#1{#2#3}{\int}$}
    \lower1ex\hbox{$#2#3$}\kern-.46\wd0}}
\def\YYYint#1#2#3{{\setbox0=\hbox{$#1{#2#3}{\int}$}
    \lower0.35ex\hbox{$#2#3$}\kern-.48\wd0}}

\def\ZZint#1#2#3{{\setbox0=\hbox{$#1{#2#3}{\int}$}
    \raise1.15ex\hbox{$#2#3$}\kern-.57\wd0}}
\def\ZZZint#1#2#3{{\setbox0=\hbox{$#1{#2#3}{\int}$}
    \raise0.85ex\hbox{$#2#3$}\kern-.53\wd0}}

\begin{document}

\renewcommand{\thefootnote}{\fnsymbol{footnote}}

\renewcommand{\thefootnote}{\fnsymbol{footnote}}

\title{ Discretization by euler's method for regular lagrangian flow.  }
\author{Christian Olivera $^{1,}$\footnote{ supported by FAPESP 
		by the grants 2017/17670-0 and 2015/07278-0, by  CNPq by the grant
		426747/2018-6.} \hskip0.2cm 
Juan D. Londo\~no $^{2}$ \vspace*{0.1in} \\
$^{1}$ Departamento de Matem\'atica, Universidade Estadual de Campinas,\\
13.081-970-Campinas-SP-Brazil. \\
colivera@ime.unicamp.br \vspace*{0.1in} \\
 $^{2}$ \quad  j209372@dac.unicamp.br \vspace*{0.1in}}

\date{}
\maketitle
\thispagestyle{empty}
\vspace{-10pt}
%%%%%%%%%%%%%%%%%%%%%%%%%%%%%%%%%%%%%%%%%%%%%%%%%%%%%%%%%%%%
\begin{abstract}

 This paper is concerned with the numerical analysis of the explicit Euler scheme
 for  ordinary differential equations with non-Lipschitz vector fields. 
We prove the convergence of the Euler scheme to regular  lagrangian flow (Diperna-Lions flows)
which is the right concept of the solution in this context. Moreover, we show that order of convergence  is $\frac{1}{2}$. 
\end{abstract}


 \smallskip


\section{Introduction}

When $b:[0,T]\times\mathbb{R}^{d}\to\mathbb{R}^{d}$ is a bounded smooth vector field, the \textit{flow of} $b$ is the smooth map $X:[0,T]\times\mathbb{R}^{d}\to\mathbb{R}^{d}$ such that
\begin{equation}\label{eqn:1}
\begin{cases}
 & \frac{dX}{dt}\left( t,x\right)=b\left( t,X\left( t,x\right)\right),\quad t\in[0,T]\\ 
 & X\left( 0,x\right)=x,
\end{cases}
\end{equation}

Out of the smooth context \eqref{eqn:1} has been studied by several authors. In particular, the following is a common definition of generalized flow for vector fields which are merely integrable.

\begin{mydef}[Regular Lagrangian flow]\label{RLF} Let $b\in L^{1}_{loc}\left([0,T]\times\mathbb{R}^{d};\mathbb{R}^{d}\right)$. We say that a map $X:[0,T]\times\mathbb{R}^{d}\to\mathbb{R}^{d}$ is a regular Lagrangian flow for the vector field $b$ if
\begin{itemize}
\item[(i)] for $m$-a.e. $x\in\mathbb{R}^{d}$ the map $t\to X\left( t,x\right)$ is an absolutely continuous integral solution of $\gamma^{\prime}\left( t\right)=b\left( t,\gamma\left( t\right)\right)$ for $t\in[0,T]$, with $\gamma\left(0\right)=x$;
\item[(ii)] there exists a constant $L$ independent of $t$ such that
\begin{equation*}
X\left( t,\cdot\right)_{\#}m\leq Lm.
\end{equation*}
The constant $L$ in $(ii)$ will be called the compressibility constant of $X$.
\end{itemize}
\end{mydef}

This paper is concerned with the numerical analysis of the euler scheme for solving ordinary differential equation (\ref{eqn:1}).
We are interested in situations in which the coefficients in the equation are rough, but
still within the range in which the associated Cauchy problem is well-posed. We now give a brief state-of-the-art survey on the theory of ordinary differential equations with vector fields of low regularity. 
In a celebrated theory established by DiPerna and Lions \cite{DL} says that
$X_{t}$ defines a regular Lagrangian flow  when $b$ is a Sobolev vector
field with bounded divergence. This theory was later extended to the case of $BV$ vector fields
by Ambrosio \cite{ambrisio}. The central of DiPerna and Lions’ theory is based on the connection between
ODE and the Cauchy problem  for the  linear transport equation. C. De Lellis and G. Crippa have recently given in \cite{Cripa} a new proof of the existence and uniqueness of the flow solution of (\ref{eqn:1}), not using the the associated transport equation. Their very interesting
approach provides regularity estimates for $W^{1,p}$ vector-fields with $p > 1$ but seemingly fails for
$W^{1,1}$ vector-fields, unfortunately.  We refer the readers to the two excellent summaries in \cite{lellis} and more recently
\cite{ambrisio2} and \cite{Jabin}.

In our  result, we show that the rate of convergence of the approximate solution given
by the explicit Euler  scheme towards the unique solution of the 
problem \label{eqn:1} is at least of order $1/2$, uniformly in time. 
 The proof is based on the  De Lellis-Crippa estimations for regular Lagrangian flow. We  mentioned that future work  we are interested in applying this scheme in PDEs with Lagrangian formulation like  transport-continuity equation \cite{Cripa}, 
 Euler equation \cite{Crippa2} and  Vlasov–Poisson system  \cite{cripa3}.

Finally we point that the classical convergence result for Euler approximation get the convergence for any initial data in $x\in \mathbb{R}^{d}$ . However  the usual 
assumptions required  differentiability and/or  Lipschitz(locally Lipschitz) regularity for the vector field $b$ see for instance 
\cite{Kloden}.  In our result (theorem  3.1)
we show the convergence in $L_{p}$ norm respect to the spatial variable which is coherent with respect to the definition of regular  Lagrangian flow.

\section{Preliminaries}

When $A$ is measurable subset of $\mathbb{R}^{d}$ we denote by $m\left( A\right)$ its Lebesgue measure. When $\mu$ is a measure on $\Omega$ and $f:\Omega\to\Omega^{\prime}$ a measurable map, $f_{\#}\mu$ will denote the push forward of $\mu$, i.e. the measure $\nu$ such that $\int\varphi\circ fd\mu=\int\varphi d\nu$ for every $\varphi\in C_{c}\left(\Omega^{\prime}\right).$\\

We recall the  following result, see for instance \cite{Cripa}.

\begin{myth}[Existence and uniqueness of the flow] Let $b$ be a bounded vector field belonging to $L^{1}\left([0,T];W^{1,p}\left(\mathbb{R}^{d};\mathbb{R}^{d}\right)\right)$ for some $p>1$. Assume that $[\mathrm{div} b]^{-}\in L^{1}\left([0,T];L^{\infty}\left(\mathbb{R}^{d}\right)\right)$. Then there exists a unique regular Lagrangian flow associated to $b$
\end{myth}

We recall here the definition of the \textit{ maximal function} of a locally finite measure and of a locally summable function and we recollect some well-known properties which are used throughout all this paper.

\begin{mydef}[ Maximal function] Let $\mu$ be a \textit{vector-field} locally finite measure. For every $\lambda>0$, we define the 
 maximal function of $\mu$ as
$$
M_{\lambda}\mu\left( x\right)=\sup_{0<r<\lambda}\frac{\lvert\mu\rvert\left( B_{r}\left( x\right)\right)}{m\left( B_{r}\left( x\right)\right)}=\sup_{0<r<\lambda}\dashint_{B_{r}\left( x\right)}d\lvert\mu\lvert\left( y\right)\quad x\in\mathbb{R}^d.
$$
When $\mu=fm$, where $f$ is a function in $L^{1}\left(\mathbb{R}^{d};\mathbb{R}^{m}\right)$, we will often use the notation $M_{\lambda}f$ for $M_{\lambda}\mu$.
\end{mydef}
The proof of the following two lemmas can be found in \cite{Stein}.
\begin{mylem}
Let $\lambda>0$. The local maximal function of $\mu$ is finite for a.e. $x\in\mathbb{R}^d$ and we have
$$
\int_{B_{\rho}(0)}M_{\lambda}f(y)dy\leq c_{d,p}+c_{d}\int_{B_{\rho+\lambda}(0)}\lvert f(y)\rvert\log\left( 2+\lvert f(y)\rvert\right)dy.
$$
For $p>1$ and $\rho>0$ we have
$$
\int_{B_{\rho}(0)}\left( M_{\lambda}f(y)\right)^{p}dy\leq c_{d,p}\int_{B_{\rho+\lambda}(0)}\lvert f(y)\rvert^{p}dy,
$$
but this is false for $p=1$.
\end{mylem}
\begin{mylem}
If $u\in BV\left(\mathbb{R}^{d}\right)$ then there exists a negligible set $N\subset\mathbb{R}^d$ such that
\[
\lvert u(x)-u(y)\rvert\leq c_{d}\lvert x-y\rvert\left( M_{\lambda}Du(x)+M_{\lambda}Du(y)\right)
\]
for $x,y\in\mathbb{R}^{d}\setminus N$ with $\lvert x-y\rvert\leq\lambda$.
\end{mylem}
%The present paper will assume $b$ satisfies the one-sided Lipschitz condition, which appeared in [Lambert]
\begin{mydef}[One-sided Lipschitz] The function $b:[0,T]\times\mathbb{R}^{d}\to\mathbb{R}^{d}$ are said to satisfy a one-sided Lipschitz condition   on $K\subset  [0,T]\times\mathbb{R}^{d}\to\mathbb{R}^{d}$ if
\begin{equation}\label{eqn:OSL}
\langle b\left( t,x\right)-b\left( t,\tilde{x}\right),x-\tilde{x}\rangle\leq\nu_{K}\left( t\right)\left| x-\tilde{x}\right|^{2}
\end{equation}
holds for a.e   in $ K$ and with $\nu \in L^{1}([0,T])$ . The function $\nu\left( t\right)$ is called a one-sided Lipschitz constant associated with $b$.
\end{mydef}

By simplicity we consider the autonomous case. We shall consider the solution $X$ of \eqref{eqn:1} given by
\begin{equation}
X\left( t,x\right)=x+\int_{0}^{t}b\left( X\left( s,x\right)\right) ds\quad\text{for }0\leq t\leq T.
\end{equation}
 Let us define an equidistant time discretization of $[0,T]$ by
\[
h=t_{n+1}-t_{n},
\]
where $h>0$ is the time step, we denote by $X_{n}$ a numerical estimate of the exact solution $X\left( t_{n}\right)$, $n=0,1,\ldots$. Then we consider the Euler scheme wich has the form
\begin{equation}
X_{n+1}=X_{n}+hb\left(X_{n}\right),\quad n=0,1,\ldots
\end{equation}
with $X_{0}=x$.

\section{Main result}

\subsection{Result}

%\cite{Kendall} \cite{Kloeden}  \cite{Ambrosio2} \cite{AmbrosioCrippa} \cite{CrippaLellis4} \cite{CrippaLecumberry} \cite{Bressan} \cite{Bressan7} \cite{Lellis} \cite{Federer} \cite{lionsbook} 
\begin{myth}\label{principal}
Let $b$ be a bounded vector field belonging to $W^{1,p}\left(\mathbb{R}^{d};\mathbb{R}^{d}\right)$ for some $p>1$, $b$ satisfies the one-sided Lipschitz condition \eqref{eqn:OSL} on any compact set and  that $[\mathrm{div}b]^{-}\in L^{\infty}\left(\mathbb{R}\right)$. Let $X$ be a regular Lagrangian flow associated to $b$, as in Definition \ref{RLF}. Then the numerical solution satisfies 
\begin{equation}
\left\| X\left( t_{n}\right)-X_{n}\right\|_{L^{p}\left( B_{R}\left( 0\right)\right)}\leq C\sqrt{\frac{C_{\exp}-1}{2\kappa}}h^{1/2}+C_{\exp}\left\| X\left( t_{0}\right)-X_{0}\right\|_{ L^{p}\left( B_{R}\left( 0\right)\right)}
\end{equation}
\end{myth}
\begin{proof}
During the proof  we use the notation $\| . \|_{L^{p}}$ for the $L^{p}$-norm in the ball  $B_{R}\left( 0\right)\subset\mathbb{R}^d$. 
By definition of  \eqref{eqn:1}  we have 

%Da Defini\c{c}\~ao \ref{RLF} temos as seguintes equa\c{c}\~oes

\begin{equation*}
X\left( t_{n+1}\right)=X_{0}+\int_{0}^{t_{n+1}}b\left(X\left( s\right)\right)ds\quad\text{e}\quad X\left( t_{n}\right)=X_{0}+\int_{0}^{t_{n}}b\left(X\left( s\right)\right)ds,
\end{equation*}
and then 
\begin{equation}\label{Xtn}
X\left( t_{n+1}\right)=X\left( t_{n}\right)+\int_{t_{n}}^{t_{n+1}}b\left(X\left( s\right)\right)ds.
\end{equation}

By definition of the Euler scheme we have 
\begin{equation}\label{EuE3}
X_{n+1}=X_{n}+hb\left(X_{n}\right).
\end{equation}

We set  

\begin{equation}\label{Yn}
Y_{n+1}:=X\left( t_{n}\right)+hb\left(X\left( t_{n}\right)\right).
\end{equation}

From \eqref{EuE3} and  \eqref{Yn} we  have 
\begin{align*}
Y_{n+1}-X_{n+1}=X\left( t_{n}\right)-X_{n}+h\left( b\left( X\left( t_{n}\right)\right)-b\left(X_{n}\right)\right),
\end{align*}

and by simple calculation we obtain 
\begin{align*}
\langle Y_{n+1}-X_{n+1},Y_{n+1}-X_{n+1}\rangle &= \langle X\left( t_{n}\right)-X_{n},Y_{n+1}-X_{n+1}\rangle\\
&\quad+h\langle b\left(X\left( t_{n}\right)\right)-b\left( X_{n}\right),Y_{n+1}-X_{n+1}\rangle.
\end{align*}

Therefore we deduce 
\begin{equation}\label{eqn:aterisco}
\begin{split}
\left| Y_{n+1}-X_{n+1}\right|^{2} &= \langle X\left( t_{n}\right)-X_{n},Y_{n+1}-X_{n+1}\rangle\\
&\quad+h\langle b\left( X\left( t_{n}\right)\right)-b\left( X_{n}\right),X\left( t_{n}\right)-X_{n}\rangle \\
&\quad\quad+h^{2}\left| b\left(X\left( t_{n}\right)\right)-b\left( X_{n}\right)\right|^{2}.
\end{split}
\end{equation}
By Cauchy-Schwarz and  Young  inequalities we  have 

\begin{equation}\label{eqn:aterisco2}
\begin{split}
\left| \langle X\left( t_{n}\right)-X_{n},Y_{n+1}-X_{n+1}\rangle\right|
&\leq \left| X\left( t_{n}\right)-X_{n}\right|\left| Y_{n+1}-X_{n+1}\right|\\
&\leq \frac{\left| X\left( t_{n}\right)-X_{n}\right|^{2}}{2}+\frac{\left| Y_{n+1}-X_{n+1}\right|^{2}}{2}.
\end{split}
\end{equation}

We observe that

\[
\|X(t_{n})\|_{\infty}\leq R + T \|b\|_{\infty}
\]

and

\[
\|X_{n}\|_{\infty}\leq R + T \|b\|_{\infty}
\]

Using that $b$ is   one-sided Lipschitz condition  on any compact set we obtain 
\begin{equation}\label{eqn:aterisco3}
\begin{split}
h\langle b\left(X\left( t_{n}\right)\right)-b\left(X_{n}\right), X\left( t_{n}\right)-X_{n}\rangle\leq\nu  h\left| X\left( t_{n}\right)-X_{n}\right|^{2},
\end{split}
\end{equation}

where $\nu$ dependent   on $T, \| b \|_{\infty}, R$. 

Now, we observe 
\begin{equation}\label{eqn:aterisco4}
\begin{split}
h^{2}\left| b\left( X\left( t_{n}\right)\right)-b\left( X_{n}\right)\right|^{2} &\leq h^{2}\left( \left| b\left(X\left( t_{n}\right)\right)\right|+\left| b\left(X_{n}\right)\right|\right)^{2}\\
&\leq4h^{2}\left\| b\right\|^{2}_{\infty}.
\end{split}
\end{equation}

From  (\ref{eqn:aterisco}), (\ref{eqn:aterisco2}), (\ref{eqn:aterisco3}), (\ref{eqn:aterisco4}) we deduce

\begin{align*}
\left| Y_{n+1}-X_{n+1}\right|^{2} &\leq\frac{\left| X\left( t_{n}\right)-X_{n}\right|^{2}}{2}+\frac{\left| Y_{n+1}-X_{n+1}\right|^{2}}{2}\\
&\quad+\kappa h\left| X\left( t_{n}\right)-X_{n}\right|^{2}+4\left\| b\right\|^{2}_{\infty}h^{2},
\end{align*}

Thus we conclude

\begin{equation}\label{eqn:a}
\left| Y_{n+1}-X_{n+1}\right|^{2}\leq\left| X\left( t_{n}\right)-X_{n}\right|^{2}\left( 1+2\kappa h\right)+8\left\| b\right\|^{2}_{\infty}h^{2},
\end{equation}

Now , taking $L^{p/2}$ in  \eqref{eqn:a} we obtain 
\begin{align*}
\left\| \left| Y_{n+1}-X_{n+1}\right|^{2}\right\|_{p/2} &=\left\| Y_{n+1}-X_{n+1}\right\|^{2}_{p}\\
&\leq\left\| \left( 1+2\kappa h\right)\left| X\left( t_{n}\right)-X_{n}\right|^{2}+8\left\| b\right\|^{2}_{\infty}h^{2}\right\|_{p/2}\\
&\leq\left\| \left( 1+2\kappa h\right)\left| X\left( t_{n}\right)-X_{n}\right|^{2}\right\|_{p/2}+\left\| 8\left\| b\right\|^{2}_{\infty}h^{2}\right\|_{p/2}\\
&\leq  \left( 1+2\kappa h\right)\left\| X\left( t_{n}\right)-X_{n}\right\|^{2}_{p}+8\left\| b\right\|^{2}_{\infty}c_{d,p}R^{d}h^{2},
\end{align*}

its implies that
\begin{equation}\label{eqn:A}
\left\| Y_{n+1}-X_{n+1}\right\|^{2}_{p}\leq \left( 1+2\kappa h\right)\left\| X\left( t_{n}\right)-X_{n}\right\|^{2}_{p}+C_{1}h^{2},
\end{equation}

with $C_{1}=8\left\| b\right\|^{2}_{\infty}c_{d,p}R^{d}$.

On other hand we have 

\begin{align*}
\left| X\left( t_{n+1}\right)-Y_{n+1}\right| &=\left| X\left( t_{n}\right)+\int_{t_{n}}^{t_{n+1}}b\left(X\left( s\right)\right) ds-X\left( t_{n}\right)-hb\left(X\left( t_{n}\right)\right)\right|\\
&=\left| \int_{t_{n}}^{t_{n+1}}\left( b\left(X\left( s\right)\right)-b\left( X\left( t_{n}\right)\right)\right) ds\right|\\
&\leq\int_{t_{n}}^{t_{n+1}}\left| b\left(X\left( s\right)\right)-b\left(X\left( t_{n}\right)\right)\right| ds.
\end{align*}

Then

\begin{equation}\label{eqn:b}
\left| X\left( t_{n+1}\right)-Y_{n+1}\right|^{2}\leq\left|\int_{t_{n}}^{t_{n+1}}\left| b\left(X\left( s\right)\right)-b\left( X\left( t_{n}\right)\right)\right| ds\right|^{2}.
\end{equation}

Taking $L^{p/2}$ in  \eqref{eqn:b} we get 

\begin{align*}
\left\|\left| X\left( t_{n+1}\right)-Y_{n+1}\right|^{2}\right\|_{p/2} &=\left\| X\left( t_{n+1}\right)-Y_{n+1}\right\|^{2}_{p}\\
&\leq\left\|\int_{t_{n}}^{t_{n+1}}\left| b\left( X\left( s\right)\right)-b\left( X\left( t_{n}\right)\right)\right| ds\right\|^{2}_{p}. 
\end{align*}

We observe that 
\begin{align*}
& \left\|\int_{t_{n}}^{t_{n+1}}\left| b\left( X\left( s\right)\right)-b\left( X\left( t_{n}\right)\right)\right| ds\right\|_{p}\\
&\leq\left\|\int_{t_{n}}^{t_{n+1}}c_{d}\left| X\left( s\right)-X\left( t_{n}\right)\right|\left( M_{\lambda}Db\left(  X\left( s\right)\right)+M_{\lambda}Db\left(X\left( t_{n}\right)\right)\right) ds\right\|_{p}\\
&\leq c_{d}\left\|\int_{t_{n}}^{t_{n+1}}\left(\int_{t_{n}}^{s}\left| b\left( X\left( u\right)\right)\right|du\right)\left( M_{\lambda}Db\left( X\left( s\right)\right)+M_{\lambda}Db\left(  X\left( t_{n}\right)\right)\right) ds\right\|_{p}\\
&\leq c_{d}\left\| b\right\|_{\infty}\left\|\int_{t_{n}}^{t_{n+1}}\left(\int_{t_{n}}^{s}du\right)\left( M_{\lambda}Db\left( X\left( s\right)\right)+M_{\lambda}Db\left( X\left( t_{n}\right)\right)\right) ds\right\|_{p}\\
&= c_{d}\left\| b\right\|_{\infty}\left\|\int_{t_{n}}^{t_{n+1}}\left( s-t_{n}\right)\left( M_{\lambda}Db\left( X\left( s\right)\right)+M_{\lambda}Db\left( X\left( t_{n}\right)\right)\right) ds\right\|_{p}\\
&\leq c_{d}\left\| b\right\|_{\infty}\int_{t_{n}}^{t_{n+1}}\left\|\left( s-t_{n}\right)\left( M_{\lambda}Db\left( X\left( s\right)\right)+M_{\lambda}Db\left( X\left( t_{n}\right)\right)\right) \right\|_{p}ds\\
&\leq c_{d}\left\| b\right\|_{\infty}\int_{t_{n}}^{t_{n+1}}\left( s-t_{n}\right)\left\| M_{\lambda}Db\left( X\left( s\right)\right)+M_{\lambda}Db\left(X\left( t_{n}\right)\right) \right\|_{p}ds\\
&\leq c_{d}\left\| b\right\|_{\infty}\int_{t_{n}}^{t_{n+1}}\left( s-t_{n}\right)\left(\left\| M_{\lambda}Db\left( X\left( s\right)\right)\right\|_{p} + \left\| M_{\lambda}Db\left(  X\left( t_{n}\right)\right) \right\|_{p}\right) ds\\
&\leq c_{d}\left\| b\right\|_{\infty}\int_{t_{n}}^{t_{n+1}}\left( s-t_{n}\right)L^{1/p}\left(\left\| M_{\lambda}Db\left( x\right)\right\|_{L^{p}\left( B_{R+T\left\| b\right\|_{\infty}}\left( 0\right)\right)}+\left\| M_{\lambda}Db\left( t_{n}, x\right) \right\|_{L^{p}\left( B_{R+T\left\| b\right\|_{\infty}}\left( 0\right)\right)}\right) ds\\
&\leq c_{d}\left\| b\right\|_{\infty}\int_{t_{n}}^{t_{n+1}}\left( s-t_{n}\right)c_{d,p}L^{1/p}\left\| Db\left( x\right)\right\|_{L^{p}\left( B_{R+\lambda + T\left\| b\right\|_{\infty}}\left( 0\right)\right) }ds\\
&\leq c_{d,p}M\left\| b\right\|_{\infty}\int_{t_{n}}^{t_{n+1}}\left( s-t_{n}\right)ds\\
&= Kh^{2},
\end{align*}
%,\quad\text{por Lema}\ref{A.4.2}

where  $K=\frac{c_{d,p}M\left\| b\right\|_{\infty}}{2}$  and we use that

\begin{align*}
\int_{t_n}^{t_{n+1}}\left( s-t_{n}\right) ds= & \left( \frac{s^{2}}{2}-st_{n}\right)_{t_{n}}^{t_{n+1}}\\
= & \left( \frac{t_{n+1}^{2}}{2}-t_{n+1}t_{n}\right)-\left( \frac{t_{n}^{2}}{2}-t_{n}^{2}\right)\\
= & \frac{t_{n+1}}{2}\left( t_{n+1}-2t_{n}\right)-\left( -\frac{t_{n}^{2}}{2}\right)\\
= & \frac{\left( t_{n}+h\right)}{2}\left( h-t_{n}\right)+\frac{t_{n}^{2}}{2}\\
= & \frac{h^{2}-t_{n}^{2}}{2}+\frac{t_{n}^{2}}{2}\\
= & \frac{h^{2}}{2}.
\end{align*}

Then  we arrive at 
\begin{equation}\label{eqn:B}
\left\| X\left( t_{n+1}\right)-Y_{n+1}\right\|^{2}_{p}\leq C_{2}h^{4}.
\end{equation}

From  H\"older and Young inequalities , \eqref{eqn:A}  and  \eqref{eqn:B} we deduce 
\begin{align*}
\left\| 2\left| X\left( t_{n+1}\right)-Y_{n+1}\right|\left| Y_{n+1}-X_{n+1}\right|\right\|_{p/2} &= 2\left(\left\|\left| X\left( t_{n+1}\right)-Y_{n+1}\right|^{p/2}\left| Y_{n+1}-X_{n+1}\right|^{p/2}\right\|_{1}\right)^{2/p}\\
&\leq 2\left(\left\|\left| X\left( t_{n+1}\right)-Y_{n+1}\right|^{p/2}\right\|_{2}\left\|\left| Y_{n+1}-X_{n+1}\right|^{p/2}\right\|_{2}\right)^{2/p}\\
&\leq 2\left\| X\left( t_{n+1}\right)-Y_{n+1}\right\|_{p}\left\| Y_{n+1}-X_{n+1}\right\|_{p}\\
&\leq 2\left( Kh^{2}\right)\sqrt{\left( 1+2\kappa h\right)\left\| X\left( t_{n}\right)-X_{n}\right\|^{2}_{p}+C_{1}h^{2}}\\
&\leq  2\left( Kh^{2}\right)\left(\sqrt{\left( 1+2\kappa h\right)\left\| X\left( t_{n}\right)-X_{n}\right\|^{2}_{p}}+\sqrt{ C_{1}h^{2}}\right)\\
&\leq 2\left( Kh^{2}\right)\left(\left( 1+2\kappa h\right)\left\| X\left( t_{n}\right)-X_{n}\right\|_{p}+\sqrt{ C_{1}}h\right)\\
&\leq 2\left( Kh^{2}\right)\left(\left( 1+2\kappa h\right)\left(\frac{\left\| X\left( t_{n}\right)-X_{n}\right\|^{2}_{p}}{2}+\frac{1}{2}\right)+\sqrt{C_{1}}h\right)
\end{align*}
%\left\| 2\left| X\left( t_{n+1}\right)-Y_{n+1}\right|\left| Y_{n+1}-X_{n+1}\right|\right\|_{p/2}
and

\[
2\left\| X\left( t_{n+1}\right)-Y_{n+1}\right\|_{p}\left\| Y_{n+1}-X_{n+1}\right\|_{p}
\]
\begin{equation}\label{eqn:C}
\leq Kh^{2}\left( 1+2\kappa h\right)\left\| X\left( t_{n}\right)-X_{n}\right\|^{2}_{p}+2KC_{3}h^{3}+Kh^{2},
\end{equation}

where $C_{3}=\kappa+\sqrt{C_{1}}$.

We have 
\begin{equation*}
\left| X\left( t_{n+1}\right)-X_{n+1}\right|\leq\left| X\left( t_{n+1}\right)-Y_{n+1}\right|+\left| Y_{n+1}-X_{n+1}\right|,
\end{equation*}

and

\begin{align*}
\left| X\left( t_{n+1}\right)-X_{n+1}\right|^{2} &\leq \left(\left| X\left( t_{n+1}\right)-Y_{n+1}\right|+\left| Y_{n+1}-X_{n+1}\right|\right)^{2}\\
&= \left| X\left( t_{n+1}\right)-Y_{n+1}\right|^{2}+2\left| X\left( t_{n+1}\right)-Y_{n+1}\right|\left| Y_{n+1}-X_{n+1}\right|+\left| Y_{n+1}-X_{n+1}\right|^{2}.
\end{align*}
Applying  $L^{p/2}$ and by \eqref{eqn:A}, \eqref{eqn:B} e \eqref{eqn:C} we obtain 
\begin{align*}
\left\| \left| X\left( t_{n+1}\right)-X_{n+1}\right|^{2}\right\|_{p/2} &=\left\| X\left( t_{n+1}\right)-X_{n+1}\right\|^{2}_{p}\\
&\leq\left\| X\left( t_{n+1}\right)-Y_{n+1}\right\|^{2}_{p}+2\left\| X\left( t_{n+1}\right)-Y_{n+1}\right\|_{p}\left\| Y_{n+1}-X_{n+1}\right\|_{p}\\
&\quad+\left\| Y_{n+1}-X_{n+1}\right\|^{2}_{p}\\
&\leq C_{2}h^{4}+Kh^{2}\left( 1+2\kappa h\right)\left\| X\left( t_{n}\right)-X_{n}\right\|^{2}_{p}+2KC_{3}h^{3}+Kh^{2}\\
&\quad+\left( 1+2\kappa h\right)\left\| X\left( t_{n}\right)-X_{n}\right\|^{2}_{p}+C_{1}h^{2}\\
&\leq C_{2}h^{2}+Kh^{2}\left( 1+2\kappa h\right)\left\| X\left( t_{n}\right)-X_{n}\right\|^{2}_{p}+2KC_{3}h^{2}+Kh^{2}\\
&\quad+\left( 1+2\kappa h\right)\left\| X\left( t_{n}\right)-X_{n}\right\|^{2}_{p}+C_{1}h^{2}\\
&= Ch^{2}+\left( 1+2\kappa h\right)\left( 1+Kh^{2}\right)\left\| X\left( t_{n}\right)-X_{n}\right\|^{2}_{p},
\end{align*}
where $C=C_{2}+2KC_{3}+K+C_{1}$.\\

Then we have 

\begin{equation}\label{En}
E_{n+1}\leq \alpha\beta E_{n}+Ch^{2}.
\end{equation}

where $\alpha=\left( 1+2\kappa h\right)$, $\beta=\left( 1+Kh^{2}\right)$ and $E_{n}=\left\| X\left( t_{n}\right)-X_{n}\right\|^{2}_{p}$.

Applying the formula  \ref{En} recursively we deduce

\begin{align*}
& E_{1}\leq\alpha\beta E_{0}+Ch^{2}\\
& E_{2}\leq\alpha\beta E_{1}+Ch^{2}\leq\alpha\beta\left(\alpha\beta E_{0}+Ch^{2}\right)+Ch^{2}\\
&\quad\,\leq\left(\alpha\beta\right)^{2} E_{0}+Ch^{2}\left( 1+\alpha\beta\right)\\
& E_{3}\leq\alpha\beta E_{2}+Ch^{2}\leq\alpha\beta\left(\left(\alpha\beta\right)^{2} E_{0}+Ch^{2}\left( 1+\alpha\beta\right)\right)+Ch^{2}\\
&\quad\,\leq\left(\alpha\beta\right)^{3} E_{0}+Ch^{2}\left( 1+\alpha\beta+\left(\alpha\beta\right)^{2}\right)\\
& E_{4}\leq\cdots.
\end{align*}

By induction we easily have that
\begin{equation}\label{eqn:En}
E_{n}\leq \left(\alpha\beta\right)^{n}E_{0}+Ch^{2}\sum_{m=0}^{n-1}\left(\alpha\beta\right)^{m}.
\end{equation}

We noted  that 

$$
\sum_{m=0}^{n-1}r^{m}=\frac{r^{n}-1}{r-1},\quad\text{para}\, r\neq 1;\quad\,\,\text{e que }\,\left| 1+zh\right|\leq\exp\left(\left| z\right| h\right)\quad\text{para }z\in\mathbb{R},
$$

\begin{align*}
\alpha^{n} &=\left( 1+2\kappa h\right)^{n}\leq\left( 1+2\kappa h\right)^{N}\\
&\leq\exp\left( 2\kappa Nh\right)\\
&=\exp\left( 2\kappa\left( T-t_{0}\right)\right),
\end{align*}
and 

\begin{align*}
\beta^{n} &\leq\left( 1+Kh^{2}\right)^{N}\leq\left( 1+Kh^{2}\right)^{N^2}\\
&\leq\exp\left( KN^{2}h^{2}\right)\\
&=\exp\left( K\left( T-t_{0}\right)^{2}\right).
\end{align*}
%\begin{align*}
%\alpha^{n} &=\left( 1+2\kappa h\right)^{n}\leq\left( 1+2\kappa h\right)^{N}\\
%&\leq\lim_{N\to\infty}\left( 1+\frac{2\kappa\left( T-t_{0}\right)}{N}\right)^{N}\\
%&=\exp\left( 2\kappa\left( T-t_{0}\right)\right),
%\end{align*}
%e
%\begin{align*}
%\beta^{n} &\leq\left( 1+Kh^{2}\right)^{N}\leq\left( 1+Kh^{2}\right)^{N^2}\\
%&\leq\lim_{\omega\to\infty}\left( 1+\frac{K\left( T-t_{0}\right)^{2}}{\omega}\right)^{\omega},\quad\text{com}\,\omega=N^2 \\
%&=\exp\left( K\left( T-t_{0}\right)^{2}\right),

Therefore we have

\begin{equation}\label{eqn:alfabeta}
\left(\alpha\beta\right)^{n}\leq\exp\left\{\left( T-t_{0}\right)\left( 2\kappa+K\left( T-t_{0}\right)\right)\right\}=C_{\exp}
\end{equation}

and 

\begin{equation}\label{eqn:alfabeta2}
\alpha\beta-1=h\left( 2\kappa+Kh+2K\kappa h^{2}\right).
\end{equation}

 From  \eqref{eqn:En},  \eqref{eqn:alfabeta} and  \eqref{eqn:alfabeta2} we conclude 
\begin{align*}
E_{n} &\leq Ch^{2}\left(\frac{\left(\alpha\beta\right)^{n}-1}{\alpha\beta-1}\right)+\left(\alpha\beta\right)^{n}E_{0}\\
&\leq Ch^{2}\left(\frac{C_{\exp}-1}{h\left( 2\kappa+Kh+2K\kappa h^{2}\right)}\right)+C_{\exp}E_{0}\\
&\leq Ch\left(\frac{C_{\exp}-1}{ 2\kappa+Kh+2K\kappa h^{2}}\right)+C_{\exp}E_{0}.
\end{align*}

Thus we have 

\begin{equation}
E_{n}\leq Ch\frac{C_{\exp}-1}{2\kappa}+C_{\exp}E_{0},
\end{equation}

and 
\begin{equation}
\left\| X\left( t_{n}\right)-X_{n}\right\|_{L^{p}\left( B_{R}\left( 0\right)\right)}\leq C\sqrt{\frac{C_{\exp}-1}{2\kappa}}h^{1/2}+C_{\exp}\left\| X\left( t_{0}\right)-X_{0}\right\|_{ L^{p}\left( B_{R}\left( 0\right)\right)}.
\end{equation}
If we take $X\left( t_{0}\right)=X_{0}=x$ then 

\begin{equation}
\left\| X\left( t_{n}\right)-X_{n}\right\|_{ L^{p}\left( B_{R}\left( 0\right)\right)}\leq C\sqrt{\frac{C_{\exp}-1}{2\kappa}}h^{1/2}=O\left( h^{1/2}\right)\quad\text{quando}\,\, h\to 0
\end{equation}
and this proof the theorem
\end{proof}

\subsection{Example}

We present one example of vector fields which satisfies the hypothesis of the theorem
(\ref{principal}). We consider $f \in W^{1,p}(\mathbb{R}^{d})$ with $p>1$ and $div f \in L^{\infty}(\mathbb{R}^{d})$. Also we consider
$g \in L^{1}(\mathbb{R}^{d})$. Now, we define $b(x)=(f\ast g)(x)$. Then by young inequality 
we have that $b \in W^{1,p}(\mathbb{R}^{d})$  and $div b \in L^{\infty}(\mathbb{R}^{d})$. We shall prove that
the function $b$ verifies one-sided Lipschitz condition on compact sets, 
we assume that $(x,y)\in K\subset \mathbb{R}^{d} \times \mathbb{R}^{d} $

\[
(b(x)-b(y)), x-y)= \int_{\mathbb{R}^{d}} g(z) (f(x-z)-f(y-z)), x-y) dz
\]

\[
= \int_{\mathbb{R}^{d}} g(z) (f(x-z)-f(y-z)), (x-z)-(y-z)) dz
\]

\[
\leq C  |x-y|^{2} \int_{\mathbb{R}^{d}} |g(z)| ( M_{\lambda}Df(x-z) + M_{\lambda}Df(y-z)   )dz
\]

\[
\leq  C |x-y|^{2} \|g\|_{1} \| M_{\lambda}Df\|_{p}
\]

\[
\leq  C |x-y|^{2} \|g\|_{1} \| DF\|_{p}\leq C |x-y|^{2}.  
\]

taking $\lambda$  big enough. 

%%%%%%%%%%%%%%%%%

\end{document}